\documentclass[12pt,a4paper]{article}
\usepackage{authblk}
\usepackage{amsmath, amssymb, amsfonts, amsthm, array}
\usepackage{graphicx, caption, subcaption, float, enumerate, diagbox,color}

\newcolumntype{P}[1]{>{\centering\arraybackslash}p{#1}}

\newtheorem{theorem}{Theorem}[section]
\newtheorem{lemma}[theorem]{Lemma}
\newtheorem{corollary}[theorem]{Corollary}

\theoremstyle{definition}
\newtheorem{definition}[theorem]{Definition}

\title{A Note on the Laplacian Eigenvectors of Threshold Graphs}

\author{Irene Sciriha\thanks{irene.sciriha-aquilina@um.edu.mt\ ORCID 0000-0002-5477-6803}}
\author{Zoia Sherman\thanks{sherman.zoya@sfa.org.ua\ ORCID 0000-0002-9050-2068}}

\author{James L. Borg\thanks{james.borg@um.edu.mt\ ORCID 0000-0002-4635-7692}}

\affil{Department of Mathematics, Faculty of Science, University of Malta, Msida MSD 2080, Malta}

\date{\today}

\begin{document}
\maketitle


\begin{center}{\textbf{Abstract}}\end{center}

Threshold graphs are graphs that can be characterized in a number of different ways. For example, they are graphs that are $P_4,\ C_4,\  2K_2$--free. They may also be characterized by a finite sequence of positive integers $a_1, \ldots, a_r$, such that $a_1\geqslant 2$ and $a_1 + a_2 + \cdots + a_r = |V(G)|$.

Threshold graphs have the remarkable property that all graphs of the same order share a common integer Laplacian eigenbasis. This property characterizes threshold graphs. This result was proved in \cite{MachareteDelVecchio}. We give a different proof of the same result. 

\noindent

\section{Introduction}

We consider finite undirected graphs without loops or multiple edges.
A graph  $G=({\mathcal V},{\mathcal E})$ has vertex set   ${\mathcal V}(G)$  and edge set ${\mathcal E}(G),$ where $| {\mathcal V}|=n$  and $|{\mathcal E}|=m$. 
The \emph{complement} $ G^C$  of a graph $G$ has the same vertex set as $G$ and its edge set consists of the non--edges of $G$. 
The complete graph on $n$ vertices is denoted $K_n$, while  $K_{1,n-1} $  denotes the star graph on $n$ vertices. 

 A \emph{clique} of a graph  $G$ is a  set of vertices that induce a maximal complete subgraph in $G.$  The same set of vertices induce a \emph{coclique} in $G^C$. 

 The \emph{adjacency matrix} of a graph $G$ is a  square matrix of size $|V(G)| \times |V(G)|$, defined by:
\[
A(G)_{ij} =
\begin{cases} 
1  & \text{if  } v_i \text{ and } v_j \text{ are adjacent } \\
0 & \text{otherwise}.
\end{cases}
\]
The \emph{degree matrix} of the graph $G$, $D(G)
$ is the diagonal matrix with $D(G)_{ii} =$ degree of vertex $v_i$. The \emph{Laplacian matrix} is defined by 
\[
L(G) = D(G) - A(G).
\]

The vertex set of the disjoint union $G_1\dot \cup G_2$ of two graphs $G_1$ and $G_2$, is  ${\mathcal V}(G_1)\cup  {\mathcal V}(G_2)$  and  its edge set is  ${\mathcal V}(G_1)\cup  {\mathcal V}(G_2). $  The join  $G_1\vee G_2$ of two graphs $G_1$ and $G_2$ is $(G_1^C\dot \cup G_2^C)^C.$ 
\emph{Threshold graphs} are graphs that can be constructed from an isolated vertex by successive disjoint unions  and/or  joins with other isolated vertices, as explained in Section 2 in more detail. The vertex set is partitioned into $r$ cells corresponding to the $r$ distance vertex degree. The \emph{creation sequence} of a threshold graph is the sequence of cell sizes.
Threshold graphs have been shown to possess  fascinating structural and spectral properties \cite{Merris, Fast, ScirihaFarrugia}.

An underlying graph of a threshold graph is the \emph{antiregular graph}, which consists of $r$ cells of size 1, except the first which is of size 2.
A threshold graphs can then be constructed from an antiregular graph with the same number of cells, by increasing the size of cells. 

Since each Laplacian eigenvalue of an antiregular graph is simple, the Laplacian eigenbasis is unique up to multiples. It is a notable property of threshold graphs, proved in \cite{MachareteDelVecchio}, that all threshold graphs of the same order share a common eigenbasis of vectors for their Laplacian eigenspaces, which is that of the antiregular graph on the same number of vertices.  We give an alternative proof of this theorem in the sequel.

The paper is organised as follows.
In Section 2, we describe threshold graphs and their construction from creation sequences.
In Section 3, we determine spectral properties of  the Laplacian matrix of threshold graphs and their associated eigenvectors, and give a new proof of the theorem mentioned above  (Theorem  \ref{eigenbasis}).

\section{Threshold Graphs}

\subsection{Basic properties}

  Threshold graphs have been discovered independently by many authors, in a number of different contexts. Initially created by Chv\'atal and Hammer  in relation to the knapsack problem \cite{ChvatalHammer},  
    threshold graphs were defined as those graphs that admit a weight function $w:{\mathcal V} \longrightarrow {\mathbb R}$ on the vertices such that
    there exists a \emph{threshold} $\xi \in {\mathbb R}$ satisfying
     $w(v_{\ell})+ w(v_k)> \xi$ if and only if 
         $\{v_{\ell}, v_k\}\in {\mathcal E}$. 
 The smallest subgraphs that do not admit a weight function are $2K_2$, $C_4$  and $P_4$. Interestingly, threshold graphs are also characterized as those graphs that are $2K_2$, $C_4$  and $P_4$ free \cite{Mahadev, GroneMerris, HammerKelmans}.   

Their construction can be encoded using a binary sequence.

\begin{definition}
 A {\rm threshold graph} on $n$ {\it vertices is defined by its} \em{binary creative sequence} {\rm (BCS)} $b_1,b_2,\ldots, b_n,$ where $b_1=b_2$.  An entry 0 in the BCS represents the  addition of an isolated vertex while an entry 1 means the addition of a dominating vertex. The entry $b_n$ is 1 for a connected threshold graph. 
\end{definition}

Starting from an isolated vertex, a given  BCS enables the construction of a unique threshold graph. The entries $b_1$ and $b_2$ are the same, and they are 1 or 0 depending on whether the first two vertices are, respectively, adjacent or not.  
A subsequence in the BCS corresponds to a clique if the entries are all equal to 1 and a coclique if they are all equal to 0.

\begin{definition}  \label{DegNSG}
 A threshold graph $C(a_1,a_2,\ldots, a_r)$ on $n$ vertices has a \emph{compact creative sequence} (CCS) 
   of positive integers, $a_1,a_2,\ldots, a_r$,  obtained from the  BCS 
 $b_1,b_2,\ldots, b_n$ of the  
  {\rm threshold graph}, such that  $\sum_{i=1}^r a_i=n,$ $a_1\geq2,$  and for $1\leq i\leq r$,  $a_i$  is the length of the $i$th subsequence of the same value (0 or 1) in BCS.
\end{definition}

It is clear that in a connected threshold graph, when starting with an isolated vertex, 
if the second vertex added is a \emph{dominating vertex}, i.e.~a vertex that is adjacent to all the other vertices of the graph, corresponding to entry $b_2=1,$ then $r$ is odd, while if $b_2=0$, then  
  $r$ is even.  The threshold graph has $r$ \emph{cells} 
  where each cell is a clique or a coclique. The degree of each vertex in a cell is the same. Therefore the number of distinct vertex degrees is the number $r$ of cells $A_1,A_2,\ldots A_r,$  of size $a_1,a_2,\ldots, a_r,$ respectively. A threshold graph is connected if and only if the subsequence of size $a_r$ in the BCS consists of 1 bits.

A threshold graph is a split graph, that is, its vertex set can be partitioned into two subsets, one inducing a 
 clique, the other a coclique. A  split graph is known to be $2K_2$, $C_4$  and $C_5$--free.  Note that $P_4$ is a split graph while $C_5$ is not a threshold graph.

\begin{definition} 
 The connected \emph{antiregular 
  graph} $C(2,1,1,\ldots,1)$ is the threshold graph  having $r$ cells with the least number of vertices. \end{definition}
  The antiregular graph has the maximum possible number of distinct vertex degrees among all connected graphs of order $n$, as it has only two vertices with the same degree.
  R. Merris 
  showed that every tree on $r+1$ vertices is a 
  subgraph of $A_{r+1}$  \cite{Merris}.  In  \cite{ScirihaFarrugia}, a threshold graph having $r$ cells is constructed from  $A_{r+1}$ by blowing up the size of selected cells to reach the cell sizes given by the CCS.  Hence we say that 
a   threshold graph with $r$ cells $A_1, \ldots A_r$ has the antiregular graph $C(2, 1,1,\ldots, 1)$ as an underlying induced subgraph also with $r$ cells.

 A threshold graph $C(a_1, \ldots, a_r), \ a_1\geq 2,$ has $r$ cells $A_1, \ldots , A_r$ that
correspond to $r$ distinct degrees.  Let $N_i$ be the open neighbourhood of each of the vertices in the coclique cell $A_i$. Note that for odd $r$, each of the $a_{2k}$ vertices in the coclique $A_{2k}$, for $1\leq k \leq \frac{r-1}{2},$
has the same  neighbourhood  $N_{2k}$ which is the disjoint union of cells $A_{2k+1}\dot \cup A_{2k+3}\dot \cup \ldots \dot \cup A_{r}.$ Therefore the  neighbourhoods are strictly nested:
  $N_{2} \supset N_{4}\supset  \ldots \supset N_{r-1}.$   Also,  each of the $a_{2k+1}$ vertices in a clique $A_{2k+1}$
has the same closed neighbourhood  $N_{2k+1}$ which is ${\mathcal V}\backslash \left (A_{r-1} \dot \cup A_{r-3} \dot \cup \ldots \dot \cup A_{2k}\right )$. For an even number of cells, the structure of a threshold graph  is similar. The structure of the coclique neighbourhoods has earned the threshold graphs, which are a subclass of split graphs, the name of \emph{nested split graphs} (NSG).
Note that we label a threshold graph  according to the construction corresponding to 
   its BCS.

    Figure \ref{NSG_schematic} shows the cells in the threshold graphs 
    $C(2,3,4,5,6,7)$ having 27 vertices, 225 edges and 6 cells, and   $C(2,1,1,5,1,3,1)$ on 14 vertices, 36 edges and 7 cells.
    
    In a graph $G$, the partition $V_1, V_2, \ldots, V_r$ of the vertex set is said to be an \emph{equitable partition} if $\bigcup_j V_j = V(G)$ and $| N(i) \cap V_j |$ is a constant for all vertices $i \in V_i$, where $j= 1, 2, \ldots, r$.
In a threshold graph, there is an equitable partition of the vertex set with each of the $r$ cells containing vertices of a unique degree.

\begin{figure}[!h]
\begin{center}
\includegraphics[trim=3cm 4.5cm 6.5cm 2cm, clip=true,width=0.8\textwidth]{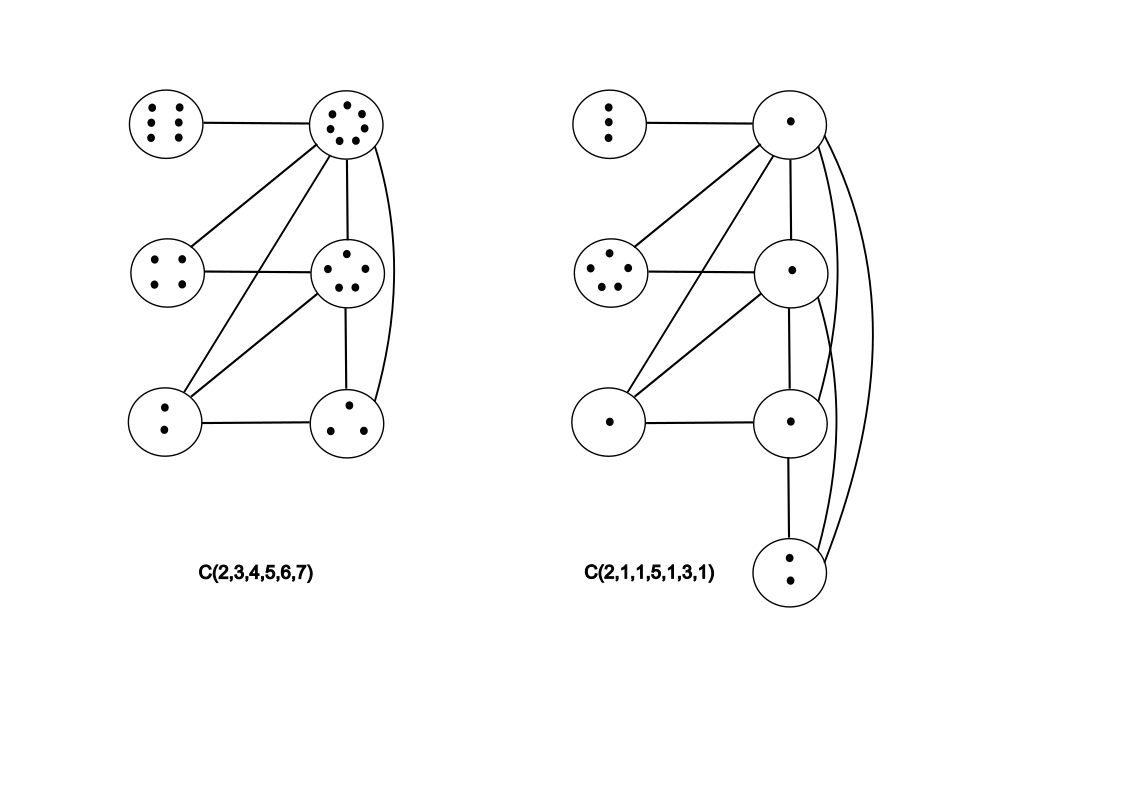}
\caption{Schematic diagrams for the threshold graphs $C(2,3,4,5,6,7)$ and $C(2,1,1,5,1,3,1)$.}  
\label{NSG_schematic}
\end{center}
\end{figure}

 From Definition \ref{DegNSG}
 and Figure \ref{NSG_schematic},  the following characterizations of  the nested neighbourhoods of the cells is immediate.

 \begin{theorem}
     \label{TheoNested}
Let $G$ be a graph with $r$ distinct degrees, where $r$ is odd, inducing an equitable vertex  partition into $r$ cells $A_i$   
such that each cell corresponds to a clique $A_{2k+1}$ or a coclique $A_{2k}$.

Then (i) the closed neighbourhoods of the clique vertices $v_i \in V_i$ satisfy  $N[v_1]\subset N[v_3]\subset \ldots \subset N[v_r] $; and

(ii)  the open neighbourhoods of coclique vertices $v_j \in V_j$ satisfy $N(v_{r-1})\subset N_(v_{r-3})\subset \ldots \subset N(v_{2}) $;

 if and only if 
     $G$ is a threshold graph.
 \end{theorem}
 For even $r$, the nesting of the neighbourhoods is similar.
 
  Theorem \ref{TheoNested}  can be restated in terms of the BCS. This gives a crucial characterization of threshold graphs in terms of their adjacency matrix with respect to the  BCS labelling.

 \begin{theorem}
    A graph $G$ with adjacency matrix     $\mathbf{A}=(a_{ij})$  is a threshold graph if and only if 
     there is a permutation of the vertex set that yields an corresponding adjacency matrix whose block form $(B_{ij})$  has alternating cliques and cocliques $B_{ii}$ on the diagonal, and  the off diagonal blocks 
     \[
     B_{i(i+1)}, B_{i(i+2)},\ldots,B_{ir}
     \]
     alternate between $\mathbf{J}$ and $\mathbf{0}$, where   $\mathbf{J}$ is the matrix with all entries equal to 1, and $\mathbf{0}$ is the matrix with all entries equal to $0$, starting with $\mathbf{J}$ if $B_{ii}$ represents a clique, and with $\mathbf{0}$ if $B_{ii}$ represents a coclique. 
      \end{theorem}
    
     Therefore, the $r\times r$ block  form of the adjacency matrix ${\bf A}=(B_{ij})$   of a threshold graph labelled according to the BCS is

      \[
      {\bf A}=(B_{ij})
= \left(
\begin{array}{ccccc}
 B_{11} & B_{12} &B_{13}  &\ldots & B_{1r}  \\
B_{21} &B_{22} & B_{23} &\ldots  & B_{2r}  \\
B_{31} &B_{32}& B_{33}&\ldots & B_{3r}\\
 \vdots &  \vdots  &  \vdots & \ddots &  \vdots
  \\
 B_{r1}  &B_{r2} &B_{r3}  &\ldots & B_{rr}\\
\end{array}
\right),
\]
where for $1\leq i\leq r$,
\[
B_{i1}= B_{i2}=\ldots  = B_{i(i-1)}=\left\{
\begin{array}{ll}
 \mathbf{0}& \hbox{if $B_{ii}$ is  a coclique}\\
    \mathbf{J}  &  \hbox{if $B_{ii}$ is  a clique.} \\
 \end{array}
 \right.
\]

\label{NSGBlockAdjacency}

Figure \ref{FigNSG27v7Cells2345678} shows the threshold graph on 27 vertices labelled according to the creation sequence $C(2,3,4,5,6,7)$,  which thus has an even number of cells. 
\begin{figure}[!h]
    \centering
    \includegraphics[trim=2cm 0cm 0cm 1.9cm, clip=true, width=\linewidth]{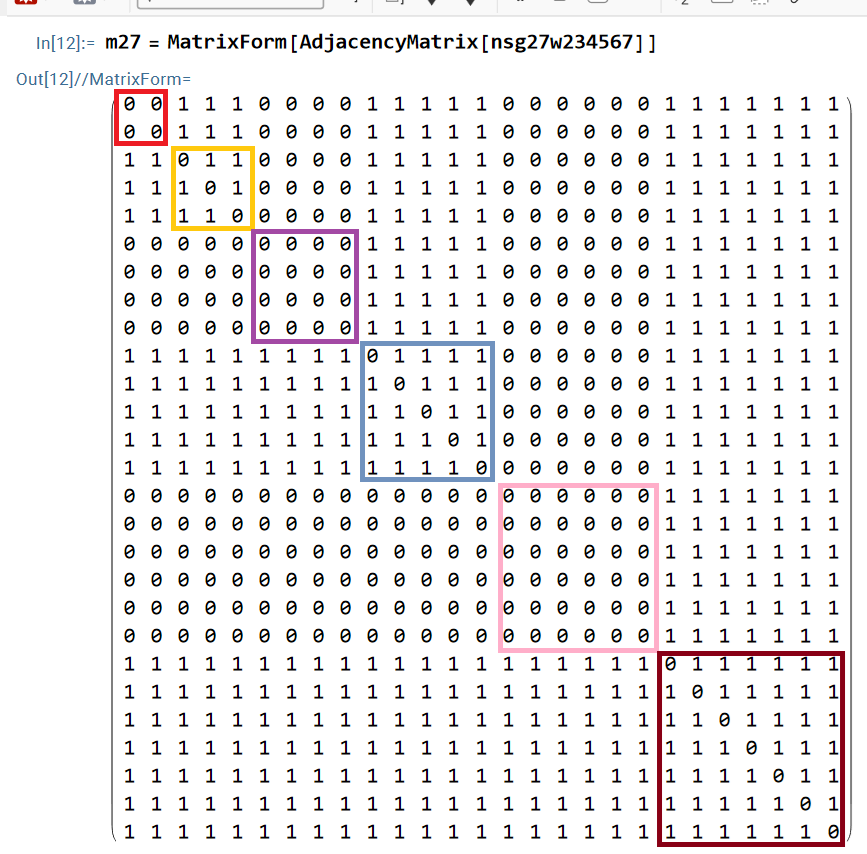}
    \caption{The adjacency matrix of the threshold graph on 27 vertices labelled according to the creation sequence $C(2,3,4,5,6,7)$. Note the block matrices on the diagonal which are alternately $\mathbf{0}$ and $\mathbf{J} - \mathbf{I}$.}
    \label{FigNSG27v7Cells2345678}
\end{figure}

Note that 
$a_{12}=0$  when $r$ is even. In this case the twin vertices 0 and 1 are duplicates, and the first diagonal block $B_{11} = \mathbf{0}$. When $r$ is odd, $B_{11} = \mathbf{J} - \mathbf{I}$.

A Ferrers--Young diagram (FYD) is used to show the partition of $2m$ as the sum of the vertex degrees of a graph. For a graph on $n$  vertices with degree sequence $d= \rho_1, \rho_2, \ldots \rho_n $,  the $i$th  row of a FYD has $\rho_i$ boxes. The conjugate degree sequence is defined to be $d^{\star}= 
\{\rho_1^{\star}, \rho_2^{\star}, \ldots, \rho_{\Delta}^{\star} \}$  where
$\rho_i^{\star}:= \left|\{j: 1\leq j\leq n \ \&  \  \rho_j \geq i\}\right|$ and $\Delta$ is the maximum degree of the graph. 

The Durfee Square (DS)  of a partition in a FYD  has size $s$ if $s$ is the largest number such that the partition of the FYD contains at least $s$ parts with values at least $s$. Equivalently, the Durfee square is the largest square that is contained within a partition's Ferrers--Young diagram.

Threshold graphs are unigraphs, that is they are determined by their degree sequence. Ruch and Gutman used FYDs to show their inverse stepwise shape for the degree sequence \cite{RuchGutman}. They proved that the degree sequence of a threshold graph majorizes the degree sequences of all graphs on the same number of edges. Merris showed that the FYD for a threshold graph has a particular shape \cite{Merris1994}, shown in Figure \ref{fig:fydshape}. 
\begin{figure}[!h]
    \centering
    \begin{subfigure}{.45\textwidth}
  \centering
   \includegraphics[trim=0cm 0cm 0cm 0cm, clip=true,width=0.8\linewidth]{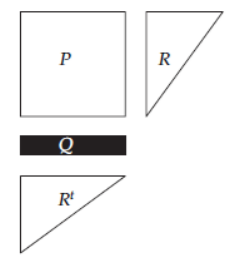}
  \caption{}
\end{subfigure}%
   \begin{subfigure}{.55\textwidth}
  \centering
    \includegraphics[trim=0cm 2cm 5cm 0cm, clip=true,width=\linewidth]{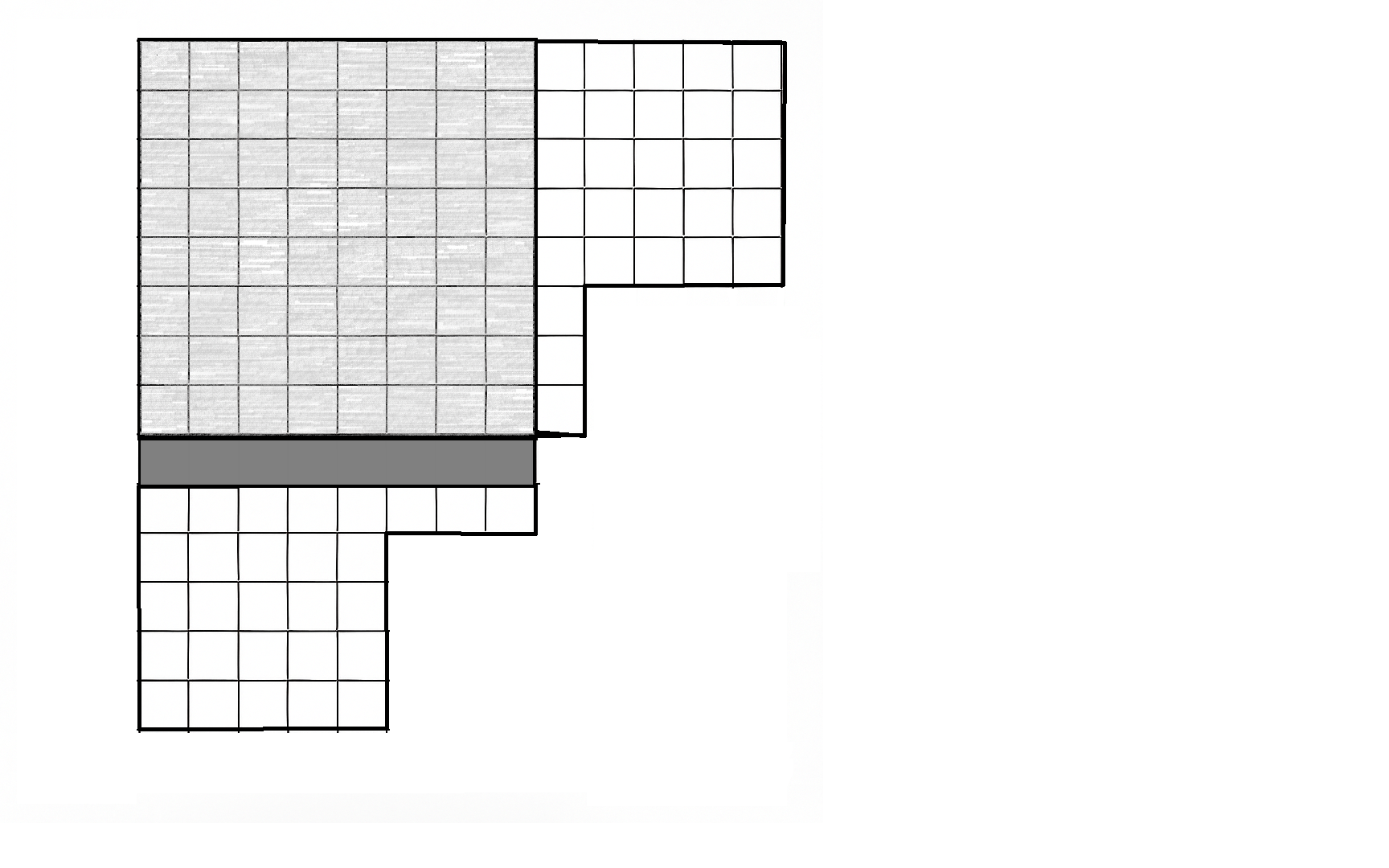}
  \caption{}
\end{subfigure}%
    \caption{Figure (a) shows the shape of the FYD for a threshold graph: $P$ is the DS size $s$, $Q$ is a row of length $s$ and $R^t$ is the transpose of shape $R$. Figure (b) is the FYD of the threshold  graph $C(2,3,4,5)$: the lengths of the rows give the degree sequence $\{ \rho_i\}$, while the lengths of the columns are the conjugate degree sequence $\{ \rho^*_i \}$.}
    \label{fig:fydshape}
\end{figure}
\begin{theorem}[\cite{Merris, Mahadev}]
    \label{Theo:chznNSG0}
 Each of  the following statements characterizes a threshold graph.
\begin{enumerate}    [(i)]
    \item Over the DS, $\rho_i^{\star} =\rho_i+1$;
    \item The shape of the FYD is as shown in Figure \ref{fig:fydshape};
    \item the graph is $P_4,\ C_4,\  2K_2$--free. 
    \item the Laplacian eigenvalues of the graph are given by 0 and the lengths $\rho_i^{\star}$.
\end{enumerate}
\label{thm:nsgchar}
\end{theorem}

\section{The Laplacian Matrix  of Threshold Graphs}

A graph is Laplacian integral if the spectrum
of its Laplacian matrix consists entirely of integers. A number
of papers on Laplacian matrices investigate the class of
Laplacian integral graphs (see \cite{HararySchwenk, Kirkland07, So99}). 
Laplacian graph eigenvalues are algebraic integers. Merris \cite{Merris1994} showed that threshold graphs are Laplacian integral. Since cographs are constructed using the join and disjoint union of smaller cographs, starting with $K_1$, it follows that they  are Laplacian integral. Hence, threshold graphs, which are a subfamily of cographs, are Laplacian integral. 

From Theorem 4.1 of \cite{ScirihaFarrugia},  the $n\times n$  adjacency matrix {\bf A} of a threshold graph $G$ is obtained from its Ferrers--Young
diagram FYD, representing the degree sequence of a $n$-vertex graph, as follows. The $i$th box is inserted
in each $i$th row and filled with a zero entry. The rest of the existing boxes are filled with the entry 1.
Boxes are now inserted so that a $n\times n$  array of boxes is obtained. Each of the remaining empty boxes is
filled with zero. The $n\times n$  array of 0-1-numbers obtained is the adjacency matrix $A$ for a labelling of the vertices according to the non--increasing degree sequence.

Note that from part (i) of Theorem \ref{thm:nsgchar}, the FYD can be constructed. The columns beyond the Durfee square correspond to the coclique vertices, and their lengths $\rho_k^{\star} = \rho_k$ give those Laplacian eigenvalues which are equal to the degree of these vertices, whereas those over the DS have lengths $\rho_k^{\star} = \rho_k+1$, as stated in (iv).

    Bai \cite{BaiLapEigs2011} proved the Grone-Merris conjecture that states that for any graph  the Laplacian eigenvalues are at most $\rho_k^{\star}$.  Moreover, as indicated in part (iv) of Theorem \ref{thm:nsgchar}, the Laplacian eigenvalues are exactly 0 and    the lengths $\rho_i^{\star}$ if and only if  the graph is  threshold. The Laplacian spectrum of a graph on $n$ vertices and the lengths $\rho_i^{\star}$ cannot agree in $n-1$ places. Kirkland \cite{Kirk2009} gave a constructive characterization of the graphs $G$ for
which the spectrum of $G$ and its degree sequence share all but two elements.

For any graph, it is well known that the degree sequence of a graph contains at least one repeated integer.

For a graph   on an even number $n$ of vertices, the monotonic non--increasing degree sequence ${n-1,n-2,\ldots,\frac{n}{2}+1, \frac{n}{2}, \frac{n}{2}, \frac{n}{2}-1, \ldots, 1}$ identifies the antiregular graph $C(2,1,1,\ldots,1)$ 
on $n$ vertices. Remarkably, as the following result shows, the Laplacian eigenvalues of an antiregular graph may be derived from the degree sequence \cite{ScirihaBorgShermanAntiregular}.

\begin{lemma}
     The Laplacian of the antiregular graph $C(2,1,1,\ldots,1)$  on $n$ vertices has $n$ distinct integer eigenvalues, $0, 1, 2, \ldots, \lceil \frac{n}{2} \rceil -1,  \lceil \frac{n}{2} \rceil +1, \ldots, n$.
 \end{lemma}
 Since the multiplicity of each eigenvalue is 1,  each eigenspace is  generated by an eigenvector that is unique up to scalar multiples. The eigenspaces are given by the following theorem.

\begin{theorem}
    The orthogonal basis    $ \left( 
      {\bf x}_{0} ,
       {\bf x}_{1},
    {\bf x}_{2},
    {\bf x}_{3},
     \ldots, {\bf x}_{n-2}, {\bf x}_{n-1}
\right)$  for ${\mathbb R}^n$,  given by the $n$   column vectors of the $n\times n$ matrix 
  \begin{equation} 
    \left(
\begin{array}{cccccccc}
 1 & -1 & -1 & -1 & \ldots  & -1 & -1 & -1 \\
 1 & 1 & -1 &  -1& \ldots & -1 & -1 & -1 \\
 1 & 0 & 2  & -1&\ldots  & -1 & -1 & -1 \\
 1 & 0 & 0 & 3  & \ldots   & -1 & -1 & -1 \\
 \vdots &  \vdots  & \vdots  & \vdots & 
\ddots
   &\vdots  &   \vdots  &   \vdots\\
 1 & 0 &0 & 0 &\ldots  & n-3 &  -1 & -1 \\
 1 & 0 & 0 &  0 & \ldots& 0 & n-2 &  -1  \\
 1 & 0 & 0 & 0 &  \ldots  & 0 & 0 &  n-1 \\
\end{array}
\right)  \label{MatrixStBas}
\end{equation}
form the only eigenbasis for the Laplacian matrix of the antiregular graph on $n$ vertices, up to sign and with mutually coprime entries in each eigenvector, associated with eigenvalues $\mu_0, \mu_1, \ldots, \mu_{n-1}$ respectively.
\end{theorem}

We note that ${\bf x}_{0} = \bf{j}$, the all-one vector (which is a kernel eigenvector of the Laplacian for any graph), while for $i = 1, \ldots, n-1$, the last nonzero entry of the eigenvector ${\bf x}_{i}$ is $i$.

\begin{definition}
    The set of $n$  column vectors of the matrix (\ref{MatrixStBas}) is termed  the {\em standard orthogonal Laplacian eigenbasis} for a graph on $n$ vertices.
 
\end{definition}

Note that if the eigenvalues are integers, then there exist eigenvectors with integer entries. We shall show in Theorem \ref{eigenbasis} that a remarkable property of threshold graphs is that they \textbf{all} share the same Laplacian eigenvectors, that is, the standard Laplacian eigenbasis,  for an appropriate labelling. This result is also proved in a recently published paper \cite{MachareteDelVecchio} (2024), where the authors show that one of the eigenbases of the star graph $K_{1, n-1}$ is also an eigenbasis for all threshold graphs on $n$ vertices. The proof we now give is different and more direct.

  \begin{theorem}
  The elements of the standard Laplacian orthogonal eigenbasis for ${\mathbb R}^n$ are an eigenbasis for the Laplacian matrix of a graph $G$ on $n$ vertices  for some labelling  of the vertex set if and only if  
   $G$ is a threshold graph. 
   \label{eigenbasis}
  \end{theorem}
\begin{proof}
   
    The Laplacian of a general graph $G$ is given by
    \[
    {\rm  Lap}(G)= \left(
\begin{array}{ccccccc}
 \rho_1 & -a_{1,2}& -a_{1,3} &\ldots  & -a_{1,k} & \ldots &-a_{1,n} \\
 -a_{1,2} &\rho_2 & -a_{2,3} &\ldots  & -a_{2,k} &  \ldots & -a_{2,n} \\
-a_{1,3} & -a_{2,3}&\rho_3 &\ldots  &-a_{3,k}&  \ldots & -a_{3,n} \\
 \vdots &  \vdots  &  \vdots & \ddots &  \vdots  &\vdots  &   \vdots  \\
 -a_{1,k} & -a_{2,k} & -a_{3,k} & \ldots &\rho_k &\ldots &  -a_{k,n}  \\
 \vdots &  \vdots &  \vdots  & \vdots&  \vdots &\ddots &  \vdots  \\
-a_{1,n}& -a_{2,n} &-a_{3,n}&\ldots  &-a_{k,n} &\vdots &  \rho_n \\
\end{array}
\right),
\]
where $a_{i,j}$ is the corresponding entry of the adjacency matrix, and $\rho_i$ is the degree of vertex $i$. We prove the theorem by induction on $i$.

Suppose that the Laplacian standard orthogonal eigenbasis
is a set of eigenvectors for the Laplacian matrix of a graph $G$ on $n$ vertices.

First, we note that ${\rm  Lap}(G). {\bf x}_0=0$, so $\mu_0 = 0.$

Now we shall consider ${\rm  Lap}(G). {\bf x}_i$ for $1 \leq i  \leq n-1$, and by induction on $i$, we  show Prop($i$) given by the following statements:

for $1\leq i\leq n-1,$ 
$a_{1,i} =a_{2,i} =\ldots
=a_{i-1,i}$, and this common value is $0$ (respectively, $1$) iff the vertex $i$ is in a coclique (respectively, a clique).

This condition is equivalent to the statement that the adjacency matrix of $G$ has the block form given in Theorem \ref{NSGBlockAdjacency}, and so it is equivalent to proving that $G$ is a threshold graph.

Consider ${\rm  Lap}(G). {\bf x}_1.$  From the first two entries, we have:
$\rho_1 + a_{12}= \rho_2 + a_{12}=\mu_2$ and so $\rho_1=\rho_2$, and $\mu_2 = \rho_1$ or $\rho_1 + 1$, if $a_{1,2} =0$ or $a_{1,2} =1$, respectively.

 From the other entries, we deduce that  $a_{1i}=a_{2i}$ for $3\leq i  \leq n$, so vertices 1 and 2 are twins (i.e.~duplicate for $\mu_1 = \rho_1$ or co-duplicate for $\mu_1 = \rho_1 +1$). This verifies Prop(1).

Now suppose that, the inductive hypothesis Prop($k$) is true for some $1 \leqslant k \leqslant n-1$, i.e.~
for $2 \leqslant j \leqslant k$, $a_{1,i} =a_{2,i} =\ldots
=a_{j,i}$ for $j+ 1 \leqslant i \leqslant n$. 
 In particular, for $j=k$ and $i=k+1$, we have $a_{1,k+1} = a_{2,k+1} = \ldots = a_{k,k+1}$. We show that Prop($k+1$) is true.

The $(k+1)$th entry of ${\rm  Lap}(G). {\bf x}_{k+1}.$ satisfies
\[
 a_{1,k+1} + a_{2,k+1} +\ldots + a_{k,k+1} + k\rho_{k+1} = k \mu_{k+1}.
\]
From the inductive hypothesis, the sum of the first $k$ terms on the LHS is $k$ or 0, so we deduce 
$\mu_{k+1}$ has the value $\rho_{k+1}$ or $\rho_{k+1} + 1$.

Consider entries $k+2, \ldots, n$ of ${\rm  Lap}(G). {\bf x}_{k+1}$.
We find that 
\[
a_{1,i} + a_{2,i} + \ldots + a_{{k},i} -  k a_{k+1, i} =0.
\]
By the inductive hypothesis, the first $k$ terms are equal, and so $a_{1,i} =a_{2,i}=\ldots
=a_{k+1,i}$, and so Prop($k+1$) is true. The result follows  by induction on $i$.
\footnote{To help the reader understand the argument better, we write out the details for $k= 2$ and $3$:

Consider ${\rm  Lap}(G). {\bf x}_3.$  From the third entry,  since $a_{13}=a_{23}$, we obtain 
$-2 \mu+3 = -2\rho_3 -a_{1,3} -a_{2,3}= -2\rho_3 - 2 a_{1,3} $. This implies that $\mu_3= \rho_3$  or $\rho_3+1$.

From the entries $4, \ldots, n$, we find that $a_{1i} + a_{2i} - 2a_{3i}=0$, and so $ a_{1i} =a_{2i} = a_{3i}$ for $4 \leq i  \leq n$.

Consider ${\rm  Lap}(G). {\bf x}_4.$  
From the fourth entry, we obtain
$-3 \mu_4 = -3\rho_4 -a_{1,4} -a_{2,4} -a_{3,4}=   -3\rho_4 - 3 a_{1,4}$, so $\mu_4= \rho_4$  or $\rho_4+1$.

As before, we also obtain $ a_{1i} =a_{2i} = a_{3i} = a_{4i}$ for $5 \leq i  \leq n$.}

We have shown that:

 Using Theorem \ref{NSGBlockAdjacency}, we have therefore proved that the graph is a threshold graph.

Conversely, suppose that $G$ is a threshold graph, labelled as in Theorem \ref{NSGBlockAdjacency}. We show that the elements of the standard Laplacian orthogonal eigenbasis are eigenvectors
of the Laplacian of $G$, where the eigenvector ${\bf x}_{i}$ corresponds to some eigenvalue $\mu_i$, for $i = 0, 1, \ldots, n-1$. Clearly,  ${\bf x}_0$ is an eigenvector of the Laplacian matrix for the eigenvalue $\mu_0 =0$.

By construction, $\rho_1 = \rho_2$,
and so the first 2 entries of ${\rm  Lap}(G). {\bf x}_{1}$ are  $-(\rho_1 +a_{1,2})$ and $\rho_1 + a_{1,2}$.
 Furthermore, since $a_{1,i}=a_{2,i}$ for $3\leq i  \leq n$,  the entries of ${\rm  Lap}(G). {\bf x}_{1}$ for these rows are all zero. This implies that ${\bf x}_1$ is an eigenvector of Lap($G$) as required, and the corresponding eigenvalue $\mu_1$ is $\rho_1$ or $\rho_1 +1$ depending on whether $a_{1,2}=0$ or $1$, respectively. This is the base step for the proof by induction.

Now suppose that ${\bf x}_k$ (for $1 \leqslant k \leqslant n-2$) is an eigenvector of Lap($G$) for the eigenvalue $\mu_k$. This is the inductive hypothesis.
By considering the first $k$ entries of this vector, we find that
\begin{align*}
  &  - \rho_1 + a_{1, 2} + a_{1, 3}  + \ldots + a_{1,k} -   k \, a_{1,k+1}
\\ & =  a_{2,1} -  \rho_2 +  a_{2, 3}  + \ldots + a_{2,k} - k \, a_{2,k+1}
\\ & = 
\ldots \\
 &  = a_{k,1} + a_{k, 2} + \ldots + a_{k, k-1}  - \rho_{k} - k \, a_{k,k+1}
 \\ & = - \mu_k,
\end{align*} 
 which we can express in matrix form as follows:
\begin{equation}
\begin{pmatrix}
- \rho_1 + a_{1, 2} + a_{1, 3}  + \ldots + a_{1,k} \\
 a_{2,1} -  \rho_2 +  a_{2, 3}  + \ldots + a_{2,k} \\
\vdots \\
a_{k,1} + a_{k, 2} + \ldots + a_{k, k-1}  - \rho_{k}
    \end{pmatrix} =
    \begin{pmatrix}
    k \, a_{1,k+1} - \mu_k \\
    k \, a_{1,k+1} - \mu_k \\
    \vdots \\
    k \, a_{1,k+1} - \mu_k
    \end{pmatrix}
    \label{eigvectormatrixeqn}
\end{equation}
 where we use the fact that in a threshold graph of order $n$,
\[
a_{i,1} =a_{i,2} =\ldots
=a_{i,i-1} \text{ for } 1\leq i\leq n  \qquad \text{ (A) }
\]
for $1\leq i\leq n$, and hence, by symmetry, $a_{1, k+1} =a_{2, k+1} =\ldots
=a_{k, k+1}$.

The $k+1$ entry is given by
\[
 a_{k+1,1} + a_{k+1, 2} + \ldots + a_{k+1, k} + k \rho_{k+1} = k \, \mu_k,
\] 
which gives
$ k \, a_{k+1,1} + k \rho_{k+1} = k \, \mu_k$, 
where we again use property (A). 
Therefore 
$\mu_k = \rho_{k+1} + a_{1,k+1}$. 
We note that this means that $\mu_k = \rho_{k+1}$ if the vertex labelled $k+1$ is a coclique vertex, while $\mu_k = \rho_{k+1} +1$ if this vertex is a clique vertex.

Let us now consider ${\rm  Lap}(G). {\bf x}_{k+1}.$ The first $k$ entries of this vector are given by
\begin{align*}
- \rho_1 + a_{1, 2} + a_{1, 3} + & \ldots + a_{1,k} + a_{1,k+1} - (k+1) \, a_{1,k+2},
\\
a_{2,1} -  \rho_2 +  a_{2, 3}  + & \ldots + a_{2,k} + a_{2, k+1} - (k +1) \, a_{2,k+2},
\\
&\vdots \nonumber \\
a_{k,1} + a_{k, 2} + a_{k, 3} + & \ldots 
- \rho_{k}  + a_{k, k+1} - (k+1) \, a_{k,k+2}.
\end{align*}
Using equation (\ref{eigvectormatrixeqn}) from the inductive hypothesis and property (A), we find that these entries are all equal to 
\begin{equation}
-\left[\mu_k -  (k+1) \, a_{1,k+1} + (k+1) \, a_{1,k+2} \right].
 \label{firstkentries}
\end{equation}
The $(k+1)$th entry is given by
\[
a_{k+1,1} + a_{k+1, 2} + a_{k+1, 3}  \ldots + a_{k+1, k}   - \rho_{k+1} - (k+1) \, a_{k+1,k+2} 
\]
which, using property (A), may be expressed as 
\begin{equation}
 - \left[\rho_{k+1}  - k \, a_{1,k+1} + (k+1) \, a_{1,k+2}\right].
 \label{k+1entry}
\end{equation}

The $(k+2)$th entry is
\[
 a_{k+2,1} + a_{k+2, 2} + a_{k+2, 3}  \ldots + a_{k+2, k+1} +  (k+1) \rho_{k+2}
\]
which reduces to
\begin{equation}
  (k+1) \left[\rho_{k+2} + a_{1,k+2}
 \right].
 \label{k+2entry}
\end{equation}
    Property (A) implies that the entries of the vector ${\rm  Lap}(G). {\bf x}_{k+1}$ in rows $k+3$ up to $n$ are zero. Therefore,  what remains is to show that the expressions given in the square brackets in (\ref{firstkentries}), (\ref{k+1entry}) and (\ref{k+2entry}) are equal.

There are four cases to consider:

Case 1. $a_{1,k+1}= a_{1,k+2} = 0$. This means that the vertices $k+1$ and $k+2$ are in the same coclique, and therefore have the same degree. Furthermore, we can also deduce that $\mu_k = \rho_{k+1}$, as explained above.
This gives the required equalities.

Case 2. $a_{1,k+1} = 0$ and $a_{1,k+2} = 1$. This means that the vertex $k+1$ is in a coclique and the vertex $k+2$ is in the following clique. This implies that $\rho_{k+2} = \rho_{k+1} + k$. We again have that $\mu_k = \rho_{k+1}$, and the required equalities follow.

Case 3. $a_{1,k+1} = 1$ and $a_{1,k+2} = 0$. This means that the vertex $k+1$ is in a clique and the vertex $k+2$ is in the following coclique. This implies that $\rho_{k+2} = \rho_{k+1} - k$, as well as that $\mu_k = \rho_{k+1} +1$. Once again, this implies the necessary equalities.

Case 4. $a_{1,k+1}= a_{1,k+2} = 1$. This means that the vertices $k+1$ and $k+2$ are in the same clique, and therefore have the same degree. Also, $\mu_k = \rho_{k+1} +1$. From this, we obtain the required equalities. 

We have thus shown that, whenever ${\bf x}_{k}$ is an eigenvector of Lap($G$), then ${\bf x}_{k+1}$ is also an eigenvector of this matrix. In all cases, we find that $\mu_{k+1} = \rho_{k+2}$ or $\rho_{k+2}  +1$. This completes the proof.
\end{proof}

The proof of Theorem \ref{eigenbasis} implies the following known result.

\begin{corollary}
The nonzero Laplacian eigenvalues of a threshold graph are either (i) $\rho_{i}$ if the vertex $i$ is in a coclique, or (ii) $\rho_{i} +1$ if the  vertex $i$ is in a clique.
The multiplicities are given by the number of vertices in the corresponding cell, except for the eigenvalue that arises from the degree of the vertices of the first cell, where the  multiplicity is one less than the size of the cell.
\end{corollary}

A well-known result from the theory of matrices is that two matrices share a basis of eigenvectors if and only if they commute \cite{StrangLA}. Hence Theorem \ref{eigenbasis} implies the following:

\begin{theorem}
For any $n \in \mathbb{Z}^+$, the space of the Laplacian matrices of threshold graphs of order $n$ is a commutative algebra.
\end{theorem}


\begin{thebibliography}{12}
	\expandafter\ifx\csname natexlab\endcsname\relax\def\natexlab#1{#1}\fi
	\providecommand{\url}[1]{\texttt{#1}}
	\providecommand{\href}[2]{#2}
	\providecommand{\path}[1]{#1}
	\providecommand{\DOIprefix}{doi:}
	\providecommand{\ArXivprefix}{arXiv:}
	\providecommand{\URLprefix}{URL: }
	\providecommand{\Pubmedprefix}{pmid:}
	\providecommand{\doi}[1]{\href{http://dx.doi.org/#1}{\path{#1}}}
	\providecommand{\Pubmed}[1]{\href{pmid:#1}{\path{#1}}}
	\providecommand{\bibinfo}[2]{#2}
	\ifx\xfnm\relax \def\xfnm[#1]{\unskip,\space#1}\fi





\bibitem{BaiLapEigs2011} H.~Bai,  The Grone-Merris conjecture, \emph{Trans. Amer. Math. Soc. }, 4463–4474.
MR2792996 ↑2.

\bibitem{GraphClasses}
A. Brandst\"adt, V.~B. Le and J.~P. Spinrad, {\it Graph classes: a survey}, SIAM Monographs on Discrete Mathematics and Applications, SIAM, Philadelphia, PA, 1999.



\bibitem{ChvatalHammer}V. Chv\'atal and P.~L. Hammer, Aggregation of inequalities in integer programming, in {\it Studies in integer programming (Proc. Workshop, Bonn, 1975)}, pp. 145--162, Ann. Discrete Math., Vol. 1, North-Holland, Amsterdam-New York-Oxford.

\bibitem{Corneil1}D.~G. Corneil, H. Lerchs and L.~S. Burlingham, Complement reducible graphs, \emph{Discrete Appl. Math.} {\bf 3} (1981), no.~3, 163--174.


\bibitem{GroneMerris}R. Grone and R. Merris, The Laplacian spectrum of a graph II, SIAM Journal on
Discrete Mathematics 7 (1994), 221-229.


\bibitem{HammerKelmans}P. Hammer and A. Kelmans, Laplacian spectra and spanning trees in threshold
graphs, Discrete Applied Mathematics 65 (1996), 255-273.

\bibitem{HararySchwenk} F.~Harary and A.~J.~Schwenk, “Which graphs have integral spectra?” in \emph{Graphs and Combinatorics: Proceedings of the Capital Conference on Graph Theory and Combinatorics at the George Washington University June 18–22, 1973}, R. A. Bari and F. Harary, Eds., vol. 406 of Lecture Notes in Mathematics, pp.45–51, Springer, Berlin, Germany, 1974.

\bibitem{Jung}H.~A.~Jung, On a class of posets and the corresponding comparability graphs, \emph{J. Combinatorial Theory}, Ser. B {\bf 24} (1978), no.~2, 125--133.

\bibitem{Kirkland07}S. Kirkland, Constructably Laplacian integral graphs, \emph{Linear Algebra and Its Applications}, vol. 423, no. 1, pp. 3–21, 2007.

\bibitem{Kirk2009} Near Threshold Graphs
Steve Kirkland. ISSN: 1077-8926
The Electronic Journal of Combinatorics (E-JC)  Volume 16, Issue 1 (2009)

\bibitem{Lerchs}Lerchs, H., \emph{On cliques and kernel}s, Tech. Report, Dept. of Comp. Sci., Univ. of Toronto (1971).

\bibitem{MachareteDelVecchio}R.~Macharete, R.~Del-Vecchio, H.~Teixeira, and L.~de Lima. A Laplacian eigenbasis for threshold graphs, \emph{Special Matrices}, vol. 12, no. 1, 2024, pp. 20240029. https://doi.org/10.1515/spma-2024-0029





\bibitem{Mahadev}  N.V.R. Mahadev and U.N. Peled, \emph{Threshold Graphs and Related Topics}, Annals of Discrete
Mathematics, (1995). Elsevier,  M09 13, ISBN: 978-0-444-89287-4.

\bibitem{Merris1994} R.~Merris, Degree maximal graphs are Laplacian integral. \emph{Linear Algebra and its Applications}, \textbf{199} (1994), 381--389.

https://doi.org/10.1016/0024-3795(94)90361-1

\bibitem{Merris1998} R.~Merris, Laplacian Graph Eigenvectors, \emph{Linear algebra and its applications}, {\bf 278} (1998), 221-–236.

\bibitem{Merris}R. Merris, {\it Graph theory}, Wiley-Interscience Series in Discrete Mathematics and Optimization, Wiley-Interscience, New York, 2001.

Threshold Graphs and Related Topics Edited by N.V.R. Mahadev - Northeastern University, Boston, MA, USA U.N. Peled - University of Illinois at Chicago, Chicago, IL, USA Pages 1-543

\bibitem{Rosa} A. Rosa, On certain valuations of the vertices of a graph, \emph{Theory of Graphs (Internat. Symposium, Rome, July 1966)}, Gordon and Breach, N. Y. and Dunod Paris
(1967) 349-355.

\bibitem{RuchGutman} E. Ruch and I. Gutman, The branching extent of graphs, J. Combin. Inform. System Sci. 4(1979)
285-295.

\bibitem{ScirihaBorgShermanAntiregular}I.~Sciriha, J.~L.~Borg and Z.~Sherman, Distinct Vertex-Edge  Spectral Labelling for Antiregular Graphs, to appear.

\bibitem{Fast}I.~Sciriha, J.~A. Briffa and M.~Debono, Fast algorithms for indices of nested split graphs approximating real complex networks, \emph{Discrete Appl. Math.} {\bf 247} (2018), 152--164.


%
\bibitem{ScirihaFarrugia} Irene Sciriha  and Stephanie Farrugia, On the Spectrum of Threshold Graphs, \emph{ISRN Discrete Mathematics}, 2011. 
DOI: 10.5402/2011/108509


\bibitem{So99}W.~So, Rank one perturbation and its application to the Laplacian spectrum of a graph, \emph{Linear and Multilinear Algebra},vol. 46, no. 3, pp. 193–198, 1999.

\bibitem{StrangLA}G.~Strang, \emph{Introduction to Linear Algebra}, 6th edition, Wellesley-Cambridge Press (2022)

\bibitem{Sumner}D.~P.~Sumner, Dacey graphs, \emph{J. Austral. Math. Soc.} {\bf 18} (1974), 492--502.



\end{thebibliography}
\end{document}